\documentclass{kms-b}

\issueinfo{}% volume number
{}%        % issue number
{}%        % month
{}%     % year
\pagespan{1}{}
\copyrightinfo{}%       % copyright year
{Korean Mathematical Society}% copyright holder
\usepackage{graphicx}
\usepackage{amsmath,amsfonts,amssymb,amsthm,version}
\usepackage{mathrsfs,fancybox,pifont}
\usepackage{float}
\allowdisplaybreaks

\theoremstyle{plain}
\newtheorem{theorem}{Theorem}[section]

\newtheorem{lemma}[theorem]{Lemma}
\newtheorem{corollary}[theorem]{Corollary}

\theoremstyle{definition}
\newtheorem*{definition}{Definition}

\theoremstyle{remark}
\newtheorem{remark}[theorem]{Remark}

\newcommand{\R}{\mathbb R}
\newcommand{\la}{\mathcal{L}^{\alpha}}
\newcommand{\nl}{\nabla_{\mathcal{L}^{\alpha}}}
\newcommand{\al}{\alpha}

\newcommand{\bv}{{\mathcal B}{\mathcal V}_{\mathcal L^{\al}}(\Omega)}
\newcommand{\bu}{{\mathcal B}{\mathcal V}_{\mathcal L^{\al}}(\Omega_1)}
\newcommand{\dv}{\mathrm{{div}_{\mathcal{L}^{\alpha}}}}
\newcommand{\di}{\mathrm{\widetilde{div}_{\mathcal{L}^{\alpha}}}}
\newcommand{\da}{d\mu_\alpha}
\newcommand{\duy}{d\mu_\alpha(y)}
\newcommand{\du}{d\mu_\alpha(x)}

\begin{document}
\title[Laguerre BV spaces, Laguerre perimeter and their applications]
{Laguerre BV spaces, Laguerre perimeter and their applications}

\author[Y. Liu]{Yu Liu}
\address{Yu Liu \\ School of Mathematics and Physics \\ University of Science and Technology Beijing \\ Beijing 100083,  China}
\email{liuyu75@pku.org.cn}

\author[H. Wang]{He Wang}
\address{He Wang \\ School of Mathematics and Physics \\ University of Science and Technology Beijing \\ Beijing 100083,  China}
\email{1329008486@qq.com}
\thanks{This work was financially supported by the National Natural Science Foundation of China (No.\,11671031), the Fundamental Research Funds for the Central Universities (No.\,FRF-BR-17-004B) and
Beijing Municipal Science and Technology Project (No.\,Z17111000220000)}

\subjclass[2020]{Primary 26A45, 46E35, 33C45}
\keywords{BV space, perimeter, isoperimetric inequality, Sobolev inequality, Laguerre operator.}

\begin{abstract}
      In this paper, we introduce the Laguerre bounded
      variation space   and the  Laguerre perimeter, thereby investigating
      their properties. Moreover, we prove the isoperimetric inequality
      and the Sobolev inequality in the Laguerre setting. As applications,
      we derive the mean curvature for the Laguerre perimeter.
\end{abstract}

\maketitle

\section{Introduction}
The spaces BV of functions of bounded variation in Euclidean spaces
have been  a classical setting now where several problems, mainly
(but not exclusively) of variational nature, find their nature
framework. For instance, when working with minimization problems,
reflexivity or the weak compactness property of the function space
$W^{1,p}(\R^{d})$ for $p>1$, the space  BV usually plays an
important role. For the case of the space $W^{1,1}(\R^{d})$, one
possible way to deal with this lack of reflexivity is to consider
the space $BV(\R^{d})$. However, the importance of generalizing the
classical notion of variation has been pointed out in several
occasions by E. De. Giorgi in \cite{EDG1}. Recently,  Huang, Li and
Liu in \cite{HLL} investigate the capacity and perimeters derived
from $\al$-Hermite bounded variation. In a general framework of
strictly local Dirichlet spaces with doubling measure, Alonso-Ruiz,
Baudoin  and  Chen et al. in \cite{PFL2} introduce the
class of bounded variation functions and proved the Sobolev
inequality under the Bakry-$\acute{\text{E}}$mery curvature type
condition. For further information on this topic, we refer the
reader to \cite{HLL1,GDP,PLa} and the references therein.

One of the aims of this paper is intended to discuss several basic
questions of geometric measure theory related to the Laguerre
operator in Laguerre BV spaces. At first, we will present a very
short introduction to the Laguerre operator.

Given a multiindex $\alpha=(\alpha_1,\cdots,\alpha_d)$, $\alpha\in
(-1,\infty)^d$, the Laguerre differential operator is defined by:
\begin{equation*}
\mathcal{L}^{\alpha}=-\sum_{i=1}^{d}\Big[x_i\frac{\partial^2}{\partial {x_i}^2}+(\alpha_i+1-x_i)\frac{\partial}{\partial{x_i}}\Big].
\end{equation*}
Consider the probabilistic gamma measure $\mu_{\alpha}$ in $\mathbb{R}_{+}^{d}=(0,\infty)^{d}$ given by
\begin{equation*}
d\mu_{\alpha}(x)=\prod_{i=1}^{d}\frac{{x_i}^{\alpha_i}e^{-x_i}}{\Gamma(\alpha_i+1)}dx:=\omega(x)dx.
\end{equation*}
It is well-known that $\la$ is positive and symmetric in $L^2(\mathbb{R}_{+}^{d},d\mu_{\alpha})$. Moreover, $\mathcal{L}^{\alpha}$ has a closure which is selfadjoint in $L^2(\mathbb{R}_{+}^{d},d\mu_{\alpha})$ and which also will be denoted by $\mathcal{L}^{\alpha}$.
We define the $i$-th partial derivative associated with $\mathcal{L}^{\alpha}$ by
\begin{equation*}
\delta_{i}=\sqrt{x_i}\frac{\partial}{\partial{x_i}},
\end{equation*}
see \cite{GIT} or \cite{Graczyk}. One of the motivations of such
definition is that
\begin{equation*}
\mathcal{L}^{\alpha}=\sum_{i=1}^{d}\delta_{i}^{*}\delta_{i},
\end{equation*}
where
\begin{equation*}
\delta_{i}^{*}=-\sqrt{x_i}\Big(\partial{x_i}+\frac{\alpha_i+\frac{1}{2}-x_i}{x_i}\Big)
\end{equation*}
is the formal disjoint of $\delta_i$ in
$L^2(\mathbb{R}_{+}^{d},\mathrm{d}\mu_{\alpha})$. Throughout this
paper, we always assume that $\Omega \subset \mathbb R_{+}^{d}$ be an open set. For
$u \in C^1(\R_{+}^d)$ and $\varphi=(\varphi_1, \varphi_2, \ldots,
\varphi_d) \in C^1(\R_{+}^d,\mathbb{R}^{d})$, we introduce the
following ${\mathcal L}^{\alpha}$-gradient operator and ${\mathcal
L}^{\alpha}$-divergence operator associated to
$\mathcal{L}^{\alpha}$:
$$\left\{
\begin{aligned}
\nabla_{\mathcal L^{\alpha}}u&:=(\delta_{1}u, \ldots,
\delta_{d}u),\\
\mathrm{div}_{\mathcal{L}^{\alpha}}\varphi&:=\delta^{*}_{1}{\varphi_1} + \delta^{*}_{2}{\varphi_2}+\cdots+\delta^{*}_{d}{\varphi_d},
\end{aligned}
\right.$$
which also gives
\begin{equation*}
\la u=\mathrm{div_{\la}}(\nl u)=-\sum_{i=1}^{d}\Big[x_i\frac{\partial^2}{\partial {x_i}^2}+(\alpha_i+1-x_i)\frac{\partial}{\partial{x_i}}\Big].
\end{equation*}

Naturally, we use $\bv$ to represent the class of all functions with
the Laguerre bounded variation ($\la$-BV in short) on $\Omega$, as a
continuation of \cite{HLL}, the goal of this paper is to consider
some related topics for the Laguerre setting, and the plan of the
notes is as follows. Section \ref{sub2.1} contains some basic facts
and notations needed in the sequel, the lower semicontinuity (Lemma
\ref{lem 1}), the completeness (Lemma \ref{complete}), the structure
theorem (Theorem \ref{Lemma 2.1.}) and approximation via
$C_{c}^{\infty}$-functions (Theorem \ref{th2.5}). It should be noted
that in contrast with Theorem 2 in \cite[Section 5.2.2]{EG}, we need
to use the mean value theorem of multivariate functions and the
intrinsic nature of the Laguerre variation. Section \ref{sub2.2} is
devoted to the perimeter $P_{\la}(\cdot,\Omega)$ induced by $\bv$,
see (\ref{def-hp}) below.

Recall that the classical perimeter of $E\subseteq \mathbb{R}^d$ is
defined by
\begin{equation*}
    P(E)=\sup_{\varphi\in\mathcal{F}(\R^d)}\left\{\int_{E}\mathrm{div}\varphi(x)dx\right\},
\end{equation*}
where $\mathcal{F}(\R^d)$ denotes the class of all functions
\begin{equation*}
    \varphi=(\varphi_{1},\cdots,\varphi_{d})\in C_c^{1}(\R^d,\R^d)
\end{equation*}
satisfying
\begin{equation*}
    \|\varphi\|_{\infty}=\sup_{x\in E}\Big\{(|\varphi_{1}(x)|^2+\cdots+|\varphi_{d}(x)|^2)^{\frac{1}{2}}\Big\}\le 1.
\end{equation*}
An elementary property of $P(E)$ is
\begin{equation}\label{eq0.1}
P(E)=P(E^c),\ \forall E\subset\R^d.
\end{equation}
In Lemma \ref{lem-5.1}, we proved that (\ref{eq0.1}) is valid for the
Laguerre perimeter $P_{\la}(\cdot)$. In section \ref{subsec3.3}, we obtain a coarea formula for $\la$-BV
functions. As an application, we deduce that the Sobolev type
inequality
\begin{equation}\label{eq2}
\|f\|_{L^{\frac{d}{d-1}}(\Omega_1, d\mu_{\alpha})}\lesssim |\nl
f|(\Omega_1)
\end{equation}
is equivalent to the following isoperimetric inequality
\begin{equation*}
{\mu_{\al}(E)}^{\frac{d-1}{d}}\lesssim
P_{\la}(E,\Omega_1),
\end{equation*}
see Theorem \ref{thm2.7}. We point out that, in the proof of
$(\ref{eq2})$, the inequality $|\nabla f(x)|(\Omega_1)\lesssim|\nl
f(x)|(\Omega_1)$ holds true. With this in mind, we consider the
subset
\begin{equation}\label{eq3}\Omega_{1}=\Omega\setminus \{x\in\mathbb{R}_{+}^d:\exists i\in
1,\cdots,d\ \mathrm{such}\ \mathrm{that}\
\sqrt{x_i}<1\}\end{equation} of $\Omega$ which is a reasonable
substitute of $\Omega$ and whose figure is given as follows:
%\begin{figure}[H]
%    \includegraphics[scale=0.38]{Picture.PNG}
%   \centering
%    \caption{Set for the Sobolev inequality in the Laguerre setting on $\R_{+}^2$}
%\end{figure}

Our motivation comes not only from the fact that these objects are
interesting on their own, but also from the possibility of their
potential applications in further research concerning the Laguerre
operator. Consequently, in Section \ref{sec-5}, we want to
investigate the Laguerre mean curvature of a set with finite
Laguerre perimeter. For the special case, i.e., the Laplace operator
$\Delta$, sets of finite perimeter were introduced by E. De Giorgi
in the 1950s, and were applied to the research on some classical
problems of the calculus of variations, such as the Plateau problem
and the isoperimetric problem, see \cite{GCP}, \cite{Giu} and
\cite{Mas}. Barozzi-Gonzalez-Tamanini \cite{Barozzi} proved that
every set $E$ of finite classical perimeter $P(E)$ in $\mathbb
R^{d}$ has mean curvature in $L^1(\mathbb R^d)$. A natural question
is that if the result of \cite{Barozzi} holds for
$P_{\la}(E,\Omega)$, $\alpha\in (-1,\infty)^{d}$. We point out that,
in the proof of main theorem of \cite{Barozzi}, the identity
(\ref{eq0.1}) is required. In Theorem \ref{thm5-1}, we generalize
the result of \cite{Barozzi} to ${P}_{\la}(\cdot,\Omega_1)$ and
prove that every set $E$ with ${P}_{\la}(E,\Omega_1)<\infty$ in
$\Omega_1$ has mean curvature in $L^{1}(\Omega_1,d\mu_{\al})$.

Throughout this article, we will use  $c$ and  $C$ to denote
positive constants, which are independent of main parameters and may
be different at each occurrence.
${\mathsf U}\approx{\mathsf V}$ indicates that
there is a constant $c>0$ such that $c^{-1}{\mathsf V}\le{\mathsf
U}\le c{\mathsf V}$, whose right inequality is also written as
${\mathsf U}\lesssim{\mathsf V}$. Similarly, one writes ${\mathsf V}
\gtrsim{\mathsf U}$ for ${\mathsf V}\ge c{\mathsf U}$.
For convenience, the positive constant $C$
may change from one line to another and this usually depends on the
spatial dimension $d$ and other fixed parameters.

\section{$\la$-BV functions}
\subsection{Fundamentals of $\la$-BV Space}\label{sub2.1}
In this section, we introduce the $\la$-BV space, i.e. the class of
all functions with the Laguerre bounded variation and investigate
its properties. The Laguerre variation ($\la$-variation in short) of
$f \in {L^1}(\Omega,d\mu_{\alpha})$ is defined by
\begin{equation*}
    |\nabla_{\la}f|(\Omega)=
    \mathop{\sup}\limits_{\varphi \in \mathcal F(\Omega)}\left\{
    \int_\Omega f(x)\dv\varphi(x)\du
    \right\},
\end{equation*}
where ${\mathcal F}(\Omega)$ denotes the class of all
functions \[\varphi=(\varphi_1,\varphi _2,\ldots
,\varphi_d) \in C_c^1(\Omega,\mathbb R^d)\] satisfying
\[ \|\varphi\|_{L^{\infty}}=\mathop{\sup}\limits_{x \in \Omega}
\Big\{(|\varphi _1(x)|^2 + \ldots + |\varphi
_d(x)|^2)^{\frac{1}{2}}\Big\} \le 1.\]

An function $f \in {L^1}(\Omega,d\mu_{\alpha})$ is said to have the $\la$-bounded variation on $\Omega$ if
$$|\nl f|(\Omega)<\infty,$$
and the collection of all such functions is
denoted by $\bv$, which is a Banach spaces with the norm
\[\|f\|_{\bv} = \|f\|_{L^{1}(\Omega,d\mu_{\al})}+|\nl f|
 (\Omega).\]

\begin{definition}\label{def-Sobolev}
    Suppose $\Omega$ is an open set in $\mathbb{R}_{+}^{d}$.
    Let $1 \le p \le \infty$. The Sobolev space $W_{\la}^{k,p}(\Omega)$ associated with $\la$ is defined as the set of all functions $f \in {L^p}(\Omega,d\mu_{\al})$ such that
    \[\delta_{j_1} \ldots  \delta_{j_m}f\in {L^p}(\Omega,\da),\ 1
    \le {j_1}, \ldots ,{j_m}\le d,\ 1\le m \le k.\]
    The norm of $f\in
    W_{\la}^{k,p}(\Omega)$ is given by
    \[\|f\|_{W_{\la}^{k,p}}:=\sum\limits_{1\le{j_1}
    \ldots {j_m} \le d,\ 1\le m \le k} {\|{\delta_{j_1}} \ldots
    \delta_{j_m}f\|}_{L^p(\Omega,d\mu_{\al})} + \|f\|_{L^p(\Omega,d\mu_{\alpha})}.\]
\end{definition}
In what follows, we will collect some properties of the space
$\mathcal{BV}_{\la}(\Omega)$.
\begin{lemma}\label{lem 1}
    \item{\rm (i)} Suppose $f \in W_{\la}^{1,1}
    (\Omega)$, then
    \begin{equation*}
    |\nabla_{\la}f|(\Omega)=\int_\Omega
    |\nabla_{\la}f(x)|\du,
    \end{equation*}which implies $W_{\la}^{1,1}
    (\Omega)\subseteq \mathcal{BV}_{\la}(\Omega).$
    \item{\rm (ii)} (Lower semicontinuity). Suppose $f_k \in \bv,\ k \in \mathbb N$ and
    $f_k \to f$ in $L_{loc}^1(\Omega,\da)$, then
    \begin{equation*}
    |\nabla
    _{\la}f|(\Omega ) \le \mathop{\lim\inf}\limits_{k \to \infty }|\nabla_{\la}f_k
    |(\Omega).
    \end{equation*}
\end{lemma}
\begin{proof}
(i) For every $\varphi\in C_c^{1}(\Omega,\mathbb R^{d})$ with
$\|\varphi\|_{L^\infty(\Omega)} \le 1$, we have
\begin{align*}
    \left|\int_{\Omega}f(x) \mathrm{div}_{\la
    } \varphi(x) \du \right|
    & = \left|\int_{\Omega} \nabla_{\la} f(x)\cdot \varphi(x)
    \du \right|\leq \int_{\Omega}|\nabla_{\la} f(x)| \du.
\end{align*}
By taking the supremum over $\varphi$, it is obvious that
$$|\nabla_{\la}f|(\Omega) \le \int_\Omega {|
\nabla_{\la}f(x)|\du}.$$

Define $\varphi\in L^\infty(\Omega,\mathbb R^{d})$ as follows:
\begin{equation}\label{eqq4}
    \varphi(x):=
    \left\{\begin{aligned}
    \frac{\nabla_{\la} f(x)}{|\nabla_{\la} f(x)|},\ &
    \text{if $x\in\Omega$ and $\nabla_{\la} f(x)\neq 0$,} \\
    0,\ & \text{otherwise.}
    \end{aligned}\right.
\end{equation}
It is easy to see that $\|\varphi\|_{L^\infty(\Omega)} \le 1$.
Moreover, we can obtain the approximating smooth fields ${\varphi
_n}:= (\varphi_{n,1}, \ldots, \varphi_{n,d})$ such that
$\varphi_n\to \varphi$ pointwise as $n \to \infty$, with
$\|\varphi_n\|_{L^\infty(\Omega)} \leq 1$ for all $n \in \mathbb N$.
Combining the definition of $|\nabla_{\la}f|(\Omega)$ with
integration by parts derives that for every $n\geq 1$,
\begin{align*}
    |\nabla_{\la}f|(\Omega)
    &\ge\int_\Omega f(x)\mathrm{div}_{\la}{\varphi_n}(x)\du\\
    &=\int_\Omega f(x)(\delta_1^{*}\varphi_{n,1}(x)+\cdots+\delta_d^{*}\varphi_{n,d}(x))\du\\
    &=\int_\Omega \nabla_{\la}f(x)\cdot{\varphi_n}(x)\du.
\end{align*}
Using the dominated convergence theorem and the definition of
$\varphi$ in (\ref{eqq4}), we have
\[|\nabla_{\la}f|(\Omega) \ge \int_\Omega
{|\nabla_{\la}f(x)|\du}\] by letting $n \to \infty$.

(ii) Fix $\varphi\in C^1_c(\Omega,\mathbb R^{d})$ with $\|\varphi
\|_{L^\infty(\Omega )} \le 1$. We use the definition of
$|\nabla_{\la}f_k|(\Omega)$ to obtain
\[|\nabla_{\la}{f_k}|(\Omega) \ge \int_\Omega{f_k}(x)\mathrm{div}_{\la}\varphi(x)\du
.\] The convergence of $\{f_k\}_{k \in \mathbb{N}}$ in
$L_{loc}^1(\Omega,\da)$ to $f$   implies that
\[\mathop{\lim\inf}\limits_{k \to \infty }|\nabla
_{\la}{f_k}|(\Omega) \ge \int_\Omega f(x)\mathrm{div}_{\la}\varphi(x)\du.\]
Therefore, (ii) can be proved by the definition of $|\nabla_
{\la}f|(\Omega)$ and the arbitrariness of such
functions $\varphi$.
\end{proof}
\begin{lemma}\label{complete}
    The space $\big(\bv,\left \|\cdot
    \right\|_{\bv}\big)$ is a Banach space.
\end{lemma}
\begin{proof}
It is easy to check that $ \left \|\cdot   \right \|_{BV_{k}
(\Omega )}  $ is a norm and we omit the details. In  what follows,
we   prove the completeness of  $\bv$. Let $ \left \{f_{n} \right
\}_{n\in \mathbb{N}} \subset \bv$ be a Cauchy sequence, namely, for
every $\varepsilon>0$, there exists $n_{0}\in \mathbb{N}$ such that
$\forall n, m\ge n_{0}$, we have$$ \left|\nabla
_{{\la}}\left(f_{m}-f_{n}\right)\right |\left(\Omega\right
)<\varepsilon.$$ Especially, $\left\{f_{n}\right\}_{n\in
\mathbb{N}}$ is a Cauchy sequence in the Banach space
$$\big(L^{1}(\Omega,\da), \|\cdot\|_{L^{1}(\Omega, \da)}\big),$$ which
implies that there exists $f\in L^{1}\left(\Omega, \da\right)$ with
$\left \|f_{n}-f\right\|_{L^{1}\left (\Omega, \da\right)} \to 0$ as
$n\to \infty$. Hence, via Lemma \ref{lem 1}(ii), we have
\begin{align*}
    \left|\nabla_{\la}(f-f_{m})\right| (\Omega)\le\liminf_{n}\left|\nabla_{\la} (f_{m}-f_{n})\right| (\Omega)\le\varepsilon,\ (\forall m\ge n_{0})
\end{align*}
which implies that  $\left |\nabla_{\la}(f_{m} -f)\right|\left
(\Omega \right)\to 0$ as $ m\to\infty$. This completes the proof.
\end{proof}

The following lemma gives the structure theorem for $\la$-BV functions and it can be proved by the Hahn-Banach theorem and the Riesz representation theorem.
\begin{lemma}(Structure Theorem for $\mathcal{BV}_{\la}$ functions).\label{Lemma 2.1.} Let $f\in \bv$. Then there exists a Radon measure $\mu_{\la}$ on $\Omega$ such that
\begin{equation*}
    \int_{\Omega } f(x) \mathrm{div}_{\la}\varphi(x)\du
    =\int_{\Omega }\varphi(x)\cdot d\mu_{\la}(x)
\end{equation*}
for every $\varphi \in C_{c}^{\infty} (\Omega,\mathbb{R}^{d})$
and
\begin{equation*}
     |\nabla_{\la}f|(\Omega)=|\mu_{\la}|(\Omega),
\end{equation*}
where $|\mu_{\la}|$ is the total variation of the measure $\mu_{\la}$.
\end{lemma}
\begin{proof}
It is easy to see that
$$\left|\int_{\Omega}f(x) \mathrm{div
}_{\la}\varphi(x)\du\right
|\le|\nabla_{\la}f|(\Omega)
\|\varphi\|_{L^\infty(\Omega)},\ \forall \varphi \in C_{c}^{\infty }(\Omega, \mathbb{R}^{d} ).$$
Denote by the functional $\Phi$ with
$$\Phi: C_{c}^\infty(\Omega,\mathbb{R}^{d})\to\mathbb{R},$$
where$$ \left\langle\Phi,
\varphi\right\rangle:=\int_{\Omega}f(x)\mathrm{div}_{\la}\varphi(x)\du.$$
Then  the Hahn-Banach theorem derives that there exists a linear and
continuous extension $L$ of $\Phi$ to the normed space $\big(
C_{c}(\Omega,\mathbb{R}^{d}),\left\|\cdot\right
\|_{L^{\infty}(\Omega)}\big)$ such that$$ \left\|L\right\|
=\left\|\Phi\right\|=|\nabla_{\la}f|\left(\Omega\right
).$$ By the Riesz representation Theorem (cf. \cite[Corollary
1.55]{AFD}), there exists a unique $\mathbb{R}^{d}$-valued finite
Radon measure $\mu_{\la}$ with
\begin{equation*}
    L(\varphi)=\int_{\Omega}\varphi(x)\cdot d\mu_{\la}(x),\ \forall \varphi \in C_{c}(\Omega, \mathbb{R}^{d})
\end{equation*}
and so that $|\mu_{\la}|(\Omega) =\|L \|.$ Thus we have
$|\mu_{\la}|(\Omega)=|\nabla_{\la}f|(\Omega)$, which completes
  the proof.
\end{proof}

In the following theorem, we can obtain the approximation
result for the $\la$-variation.
\begin{theorem}\label{th2.5}
Let $\Omega_1$ be an open set defined in (\ref{eq3}).
Assume that $u\in\bu$, then there exists a sequence
$\{u_h\}_{h\in\mathbb{N}}\in\bu \cap C_c^\infty (\Omega_1)$ such that
$$\mathop {\lim}\limits_{h \to \infty}\|u_h-u\|_{L^1(\Omega_1,
d\mu_{\alpha})} = 0$$ and
\[\mathop {\lim}\limits_{h \to \infty} \int_{\Omega_1} |\nl u_h(x)|\du=|\nl u|(\Omega_1).\]
\end{theorem}
\begin{proof}
We adapt the method of the proof in \cite[Section 5.2.2, Theorem
2]{EG}, but different from its proof,
 we need to use the mean value
theorem of multivariate functions and  the intrinsic nature of the
$\la$-variation. Via the lower semicontinuity of ${\mathcal
L}^{\alpha}$-BV functions, we only need to show that for
$\varepsilon>0$, there exists a function $u_\varepsilon\in
C^\infty(\Omega_1)$ such that
\begin{equation*}
    \int_{\Omega_1} |{u_\varepsilon}(x)-u(x)|\du<\varepsilon
\end{equation*}
and
\begin{equation*}
    |{\nl}u_\varepsilon|(\Omega_1)\le|\nl u|(\Omega_1)+\varepsilon.
\end{equation*}

Fix $\varepsilon>0$. Given a positive integer $m$, define a sequence of open sets,
\[{\Omega_{1,j}}:=\Big\{x \in \Omega_1:\mathrm{dist}(x,\partial\Omega_1)>
\frac{1}{m+j}\Big\} \cap B(0,m + j),\ j\in\mathbb{N},\] where $\mathrm{dist}(x,\partial\Omega_1)=\inf\{|x-y|:y\in \partial \Omega_1\}$. Note that ${\Omega_{1,j}} \subset {\Omega_{1,j + 1}} \subset \Omega_1, \ j\in\mathbb{N}$ and $\mathop\cup\limits_{j = 0}^\infty{\Omega_{1,j}}=\Omega_1$. Since $|\nl u|(\cdot)$ is a measure, then choose a
$m\in\mathbb{N}$ so large such that
\begin{equation}\label{eq-2.2}
 |\nl u|(\Omega_1\backslash\Omega_{1,0})<\varepsilon.
\end{equation}

Set ${U_0}:=\Omega_{1,0}$ and $U_j:=\Omega_{1,j+1}\backslash
{\overline\Omega}_{1,j-1}$ for $j\ge1$. By standard results from \cite[Section 5.2.2, Theorem 2]{EG}, we conclude that there is a partition of unity associated to the covering $\{U_j\}_{j\in\mathbb N}$. Namely, there exist functions $\{f_j\}_{j\in\mathbb N} \in C_c^\infty(U_j)$ such that $0\le f_j
\le 1,\ j\ge0$ and $\sum\limits_{j=0}^\infty f_j=1$ on
$\Omega_1$.
Thus we have the fact that
\begin{align}\label{eq1}
    \sum_{j=0}^{\infty}\nl f_j&=\Big(\sqrt{x_1}\frac{\partial}{\partial{x_1}}\Big(\sum_{j=0}^{\infty} f_j\Big),\sqrt{x_2}\frac{\partial}{\partial{x_2}}\Big(\sum_{j=0}^{\infty} f_j\Big),\cdots,\sqrt{x_d}\frac{\partial}{\partial{x_d}}\Big(\sum_{j=0}^{\infty} f_j\Big)\Big)\\
    &=0\nonumber
\end{align}
on $\Omega_1$. Given $\varepsilon>0$ and $u \in {L^1}(\Omega_1
,\mathbb{R})$, extended to zero out of $\Omega_1$, we define the
following regularization
\[{u_\varepsilon}(x):=\frac{1}{\varepsilon^d}
\int_{B(x,\varepsilon)} \eta\Big(\frac{x-y}{\varepsilon
}\Big)u(y)\duy,\] where $\eta\in C_c^\infty (\mathbb{R}_{+}^d)$ is a
nonnegative radial function satisfying
$\frac{1}{\epsilon_j^d}\int_{\mathbb{R}_{+}^d}\eta\big(\frac{x-y}{\epsilon_j}\big)\du=1,\ \forall\ j\in\mathbb{N}$, and $\text{supp}\ \eta \subset B(0,1)\cap\R_{+}^d$. Then for
each $j$, there exists $0<\varepsilon_j<\varepsilon$ so small such that
\begin{equation}\label{eq-2.1}
     \left\{
    \begin{aligned}
    \hspace{-0.5cm}
       {\rm{supp}}({({{f_j}u})}_{\varepsilon _j}
        )\subseteq {U_j},\\
       \int_{\Omega_1} |({f_j}u)_{\varepsilon_j}(x)-{f_j}u(x)
        |\du<\varepsilon 2^{-(j + 1)},\\
       \;\int_{\Omega_1}  |(u\nl {f_j})_{\varepsilon_j}(x)-
       u\nl {f_j}(x)|\du <\varepsilon 2^{-(j + 1)}.
    \end{aligned}
\right.
\end{equation}

Construct \[{v_\varepsilon}(x):=\sum\limits_{j=0}^\infty
{(u{f_j})}_{\varepsilon _j}(x).\] In some neighborhood of each point $x\in \Omega_1$, there are only finitely many nonzero terms in this sum, hence $v_\varepsilon
\in C^\infty(\Omega_1)$ and $u = \sum\limits_{j = 0}^\infty
{u{f_j}}$. Therefore, by a simple computation, we obtain
\begin{equation*}
    \|v_\varepsilon-u\|_{L^1(\Omega_1,d{\mu}_{\alpha})} \le
    \sum\limits_{j=0}^\infty \int_{\Omega_1} |{({f_j}u
    )}_{\varepsilon_j}(x)-{f_j}(x)u(x)|\du<\varepsilon.
\end{equation*}
Consequently,
\begin{equation*}
    v_\varepsilon \to u \ \ \mathrm{in} \ \
    {L^1}(\Omega_1,d\mu_{\al})\ \ \mathrm{as} \ \ \varepsilon\to 0.
\end{equation*}
Now, assume $\varphi\in C_c^1(\Omega_1,\mathbb{R}^d)$ and
$|\varphi|\le 1$. We decompose the  integral as  follows:
\begin{align*}
    &\int_{\Omega_1} v_\varepsilon(x)\mathrm{div}_{{\mathcal L}^{\alpha}}\varphi(x)\du\\
    &=\int_{\Omega_1}\Big(\sum\limits_{j=0}^\infty  {(u{f_j})}_
    {\varepsilon_j}(x)\Big)\mathrm{div}_{{\mathcal L}^{\alpha}}\varphi(x)\du\\
    &=\sum\limits_{j=0}^\infty\int_{\Omega_1} {(u{f_j})}_{
    \varepsilon_j}(x)\Big({\delta^{*}_{1}}{\varphi_1}(x) + {\delta^{*}_{2}}{\varphi_2}(x) +\cdots+{\delta^{*}_{d}}{\varphi_d}(x)\Big)\du\\
    &:=I+II,
\end{align*}
where
\[\left\{
\begin{aligned}
    I&:=-\sum\limits_{j=0}^\infty\int_{\Omega_1}{(u{f_j})}_{\varepsilon_
        j}(x)\Big(\sqrt{x_1}\frac{\partial}{\partial{x_1}}{\varphi_1}(x)+\cdots+
    \sqrt{x_d}\frac{\partial}{\partial x_d}{\varphi_d}(x)\Big)\du,\\
    II&:=-\sum\limits_{j=0}^\infty \int_{\Omega_1}
    (uf_j)_{\varepsilon_j}(x)
    \Big(\frac{\alpha_1+\frac{1}{2}-x_1}{\sqrt{x_1}}\varphi_1(x)+
    \cdots+
    \frac{\alpha_d+\frac{1}{2}-x_d}{\sqrt{x_d}}\varphi_d(x)\Big)\du.
\end{aligned}\right.\]

For the sake of research, let
\begin{equation}\label{eqq7}
\mathrm{\widetilde{div}}_{\mathcal{L}^{\alpha}}\varphi=\delta_{1}{\varphi
_1}+\delta_{2}{\varphi_2}+ \cdots+\delta_{d}{\varphi_d}.
\end{equation}

As for $I$, we obtain
 \begin{align*}
    I&=-\sum\limits_{j=0}^\infty \int_{\Omega_1} {(u{f_j})}_{\varepsilon_j
    }(x)\mathrm{\widetilde{div}}_{{\mathcal{L}}^{\alpha}}\varphi(x)\du\\
    &=-\sum\limits_{j=0}^\infty \int_{\Omega_1}  (u{f_j})(y)\mathrm{\widetilde{div}_{{\mathcal L}^{\alpha}}}(\eta_{\varepsilon_j}*\varphi(y))\duy\\
    &=-\sum\limits_{j=0}^\infty \int_{\Omega_1} u(y)\di({f_j}(\eta_
    {\varepsilon _j}*\varphi))(y)\duy\\&\quad+\sum\limits_{j=0}^\infty\int_{
        \Omega_1} u(y)\nl{f_j}\cdot(\eta_{\varepsilon_j}*\varphi)(y)\duy\\
    &=-\sum\limits_{j=0}^\infty \int_{\Omega_1}u(y)\di({f_j}(\eta_
    {\varepsilon_j}*\varphi))(y)\duy\\&\quad\sum\limits_{j=0}^\infty \int_{\Omega_1}
    \varphi(y)\Big(\eta_{\varepsilon_j}*(u\nl {f_j})(y)-u\nl {f_j}
    (y)\Big)\duy\\
    &:={I_1}+{I_2},
\end{align*}
where in the last equality we have used the fact (\ref{eq1}). In fact, when
$\|\varphi\|_{L^\infty}\le 1$, it holds that $|f_j(\eta_{\varepsilon_j}
*\varphi)(x)|\le 1,\ j\in\mathbb{N}$, and each point in $\Omega$ belongs
to at most three of the sets $\{U_j\}_{j=0}^\infty$. Furthemore,
(\ref {eq-2.1}) implies that $|I_2|<\varepsilon$.

On the other hand, we change the order of integration to get
\begin{align*}
    II&=-\sum\limits_{j=0}^{\infty}\int_{\Omega_1}  {(u{f_j})}_{\varepsilon_
        j}(x)\left(\frac{\alpha_1+\frac{1}{2}-x_1}{\sqrt{x_1}}\varphi_1(x)+\cdots+
    \frac{\alpha_d+\frac{1}{2}-x_d}{\sqrt{x_d}}\varphi_d(x)\right) \du\\
    &=-\sum\limits_{j=0}^\infty \int_{\Omega_1}\int_{\Omega_1}\frac{1}{\varepsilon_j^d}\eta \Big( \frac{x-y}{\varepsilon_j}\Big)u(y){f_j}(y)\Big(\sum\limits_{k=1}^d \frac{\alpha_k+\frac{1}{2}-y_k}{\sqrt{y_k}}{\varphi_k}(x)\Big)\duy\du\\
    &\quad-\sum\limits_{j=0}^{\infty}\int_{\Omega_1}\int_{\Omega_1}\frac{1}
    {\varepsilon_j^d}\eta\Big(\frac{x-y}{\varepsilon_j}
    \Big)u(y){f_j}(y)\\ &\quad\times\left(\sum\limits_{k = 1}^d \Big(\frac{\alpha_k+\frac{1}{2}-x_k
    }{\sqrt{x_k}}-\frac{\alpha_k+\frac{1}{2}-y_k}{\sqrt{y_k}}\Big){\varphi_k}(x)   \right)\duy\du\\
    &=-\sum\limits_{j=0}^\infty\int_{\Omega_1} u(y){f_j}(y) \Big(\Big(\sum
    \limits_{k=1}^{d} \frac{\alpha_k+\frac{1}{2}-y_k}{\sqrt{y_k}}{\varphi_k}\Big)\ast\eta _{\varepsilon_j}(y)\Big)\duy\\
    &\quad-\sum\limits_{j=0}^{\infty}\int_{\Omega_1}\int_{\Omega_1}\frac{1}
    {\varepsilon_j^d}\eta \Big(\frac{x-y}{\varepsilon_j}
    \Big)u(y){f_j}(y)\\ &\quad\times\left(\sum\limits_{k=1}^d\Big(\frac{\alpha_k+\frac{1}{2}-x_k
    }{\sqrt{x_k}}-\frac{\alpha_k+\frac{1}{2}-y_k}{\sqrt{y_k}}\Big){\varphi_k}(x)\right)\duy\du.
\end{align*}

Thus, the above estimate of the term $I_2$ shows that
\begin{equation*}
    \left|\int_{\Omega_1}{v_\varepsilon}(x)\mathrm{div}_{{\mathcal L}^{\alpha}
    }\varphi(x)\du\right|=|I_1+I_2+II|\le J_1+J_2+\varepsilon,
\end{equation*}
where
\begin{align*}
    J_1&:=\Bigg|-\sum\limits_{j=0}^\infty\int_{\Omega_1} u(y)\mathrm{\widetilde{div}}(f_j
    (\eta_{\varepsilon_j}*\varphi))(y)\duy\\
    &\quad-\sum\limits_{j=0}^\infty\int_{\Omega_1} u(y){f_j}(y) \left(\sum\limits_{k=1}^d
    \frac{\alpha_k+\frac{1}{2}-y_k}{\sqrt{y_k}}({\varphi_k}*\eta_{
        \varepsilon_j}(y))\right)\duy\Bigg|
\end{align*}
and
\begin{align*}
    J_2:&=\Bigg|-\sum\limits_{j=0}^\infty\int_{\Omega_1}\int_{\Omega_1}
    \frac{1}{\varepsilon_j^d}\eta\Big(\frac{x-y}{\varepsilon
        _j}\Big)u(y){f_j}(y)\\ &\times\left(\sum\limits_{k=1}^d\Big( \frac{\alpha_k+\frac{1}{2}-x_k}{\sqrt{x_k}}-\frac{\alpha_k+\frac{1}{2}-y_k}{\sqrt{y_k}}\Big ){\varphi _k}(x)\right)\duy\du\Bigg|.
\end{align*}
Furthermore,
\begin{align*}
    J_1&=\Bigg|-\sum\limits_{j=0}^\infty\int_{\Omega_1}u(y)\di(f_j
    (\eta_{\varepsilon_j}*\varphi))(y)\duy\\&\quad-\sum\limits_{j=0}^\infty\int_{\Omega_1} u(y){f_j}(y)\left( \sum\limits_{k=1}^d\frac{\alpha_k+\frac{1}{2}-y_k}{\sqrt{y_k}} \varphi_k*\eta_{\varepsilon _j}(y)\right)\duy\Bigg|\\
    &\le\Bigg|-\int_{\Omega_1}u(y)\di({f_0}(\eta_{\varepsilon_0}*
    \varphi))(y)\duy\\&\quad-\int_{\Omega_1}u(y){f_0}(y) \left(\sum\limits_{k=1}^d
    \frac{\alpha_k+\frac{1}{2}-y_k}{\sqrt{y_k}} {\varphi_k}*
    \eta_{\varepsilon_0}(y)\right)\duy\Bigg|\\
    &\quad+\Bigg|-\sum\limits_{j=1}^\infty\int_{\Omega_1}u(y)\di
    (f_j(\eta_{\varepsilon_j}*\varphi ))(y)\duy\\&\quad-\sum\limits_
    {j=1}^\infty\int_{\Omega_1}u(y){f_j}(y)\left(\sum\limits_{k=1}^d
    \frac{\alpha_k+\frac{1}{2}-y_k}{\sqrt{y_k}} {\varphi_k}*
    \eta_{\varepsilon_j}(y)\right)\duy\Bigg|\\
    &\le |\nabla_{{\mathcal L}^{\alpha}}u|(\Omega_1)+
    \sum\limits_{j=1}^\infty |\nabla_{{\mathcal L}^{\alpha}}u
    |(U_j)\\
    &\le |\nabla_{\mathcal L^{\alpha}}u|(\Omega_1)+
    |\nabla_{{\mathcal L}^{\alpha}}u|(\Omega_1\backslash
    \Omega_{1,0})\\
    &\le |\nabla_{\mathcal L^{\alpha}}u|(\Omega_1)+3\varepsilon,
\end{align*}
where we have used the fact (\ref{eq-2.2}) in the last
inequality. Note that $\psi(x_k)=\frac{\alpha_k+\frac{1}{2}-x_k}{\sqrt{x_k}}$, $\|
\varphi\|_{L^\infty}\le 1$ and $\text{supp}\ \eta\subseteq B(0,1)\cap\R_{+}^d$. When $|x_k-y_k|<{\varepsilon_j}<|y_k|/{2}$, by the mean value theorem of multivariate functions, there exists $\theta\in (0,1)$ such that
\begin{align*}
    |\psi(x_k)-\psi(y_k)|&=\Big|\frac{\alpha_k+\frac{1}{2}}{2}(y_k+\theta(x_k-y_k)   )^{-\frac{3}{2}}+\frac{1}{2}(y_k+\theta(x_k-y_k))^{-\frac{1}{2}}\Big|
    |x_k-y_k|\\
    &\le\Big(\frac{|\alpha_k+\frac{1}{2}|}{2}|y_k+\theta(x_k-y_k)|^{-\frac{3}{2}}    +\frac{1}{2}|y_k+\theta(x_k-y_k)|^{-\frac{1}{2}}\Big)|x_k-y_k|.
\end{align*}

Consequently, we obtain
\begin{align*}
    J_2&=\Bigg|\sum\limits_{j=0}^\infty\int_{\Omega_1}\int_{\Omega_1}
    \frac{1}{\varepsilon_j^d}\eta\Big(\frac{x-y}{\varepsilon_j}
    \Big)u(y){f_j}(y)\Bigg(\sum\limits_{k=1}^d \Big(\frac{\alpha_k+\frac{1}{2}-x_k
    }{\sqrt{x_k}}-\frac{\alpha_k+\frac{1}{2}-y_k}{\sqrt{y_k}}\Big){\varphi_k}(x)\Bigg)\\&
    \quad\times\duy\du\Bigg |\\
    &\le\frac{\varepsilon_j}{2}\sum_{j=0}^{\infty}\int_{\Omega_1}
    \int_{\Omega_1}\Big|\frac{1}{\varepsilon_j^d}\eta\Big(\frac{x-y}{\varepsilon_j}
    \Big)u(y){f_j}(y)\Big|\sum_{k=1}^{d}\Big|\alpha_k+\frac{1}{2}\Big||y_k+\theta(x
    -y_k)|^{-\frac{3}{2}}\\&\quad\times|\varphi_k(x)|\duy\du\\
    &\quad+\frac{\varepsilon_j}{2}\sum_{j=0}^{\infty}\int_{\Omega_1}\int_{\Omega_1}\Big|  \frac{1}{\varepsilon_j^d}\eta\Big(\frac{x-y}{\varepsilon_j}
    \Big)u(y){f_j}(y)\Big|\sum_{k=1}^{d}|y_k+\theta(x_k-y_k)|^{-\frac{1}{2}}\\ &\quad\times|\varphi_k(x)|\duy\du\\
    &\le C\varepsilon_j\sum_{j=0}^{\infty}\int_{\Omega_1}\int_{\Omega_1}
    \Big|\frac{1}{\varepsilon_j^d}\eta\Big(\frac{x-y}{\varepsilon_j}
    \Big)u(y){f_j}(y)\Big|\sum_{k=1}^{d}\Big|\alpha_k+\frac{1}{2}\Big||y_k|^{-\frac{3}{2}}\duy\du\\
    &\quad+C\varepsilon_j\sum_{j=0}^{\infty}\int_{\Omega_1}\int_{\Omega_1}\Big|  \frac{1}{\varepsilon_j^d}\eta\Big(\frac{x-y}{\varepsilon_j}
    \Big)u(y){f_j}(y)\Big|\sum_{k=1}^{d}|y_k|^{-\frac{1}{2}}\duy\du\\
    &\le C\varepsilon_j\sum\limits_{j=0}^\infty \int_{\Omega_1}
    \int_{\Omega_1}  \Big|\frac{1}{\varepsilon_j^d}\eta \Big(\frac{x-y}
    {\varepsilon_j}\Big)\Big|\du\sum_{k=1}^{d}\Big|\alpha_k+\frac{1}{2}\Big| |u(y)|  |{f_j}(y)|{|y_k|}^{-\frac{3}{2}}\duy\\
    &\quad+C\varepsilon_j\sum\limits_{j=0}^\infty\int_{\Omega_1}
    \int_{\Omega_1}\Big|\frac{1}{\varepsilon_j^d}\eta\Big(\frac{x-y}
    {\varepsilon_j}\Big)\Big|\du\sum_{k=1}^{d}|u(y)||f_j(y)||y_k|^{-\frac{1}{2}}\duy\\
    &\le C\varepsilon_j\int_{\mathbb{R}_{+}^d}\Big|\frac{1}{\epsilon_j^d}\eta\Big(\frac{x-y}{\epsilon_j}\Big)\Big|\du
    \sum\limits_{k=1}^{d} \int_{\Omega_1}\Big|\alpha_k+\frac{1}{2}\Big||u(y)|\Big| \sum_{j=0}^{\infty}{f_j}(y)\Big||y_k|^{-\frac{3}{2}}\duy\\
    &\quad+C\varepsilon_j\int_{\mathbb{R}_{+}^d}\Big|\frac{1}{\epsilon_j^d}\eta\Big(\frac{x-y}{\epsilon_j}\Big)\Big|\du
    \sum\limits_{k=1}^{d}\int_{\Omega_1}|u(y)|\Big|\sum_{j=0}^{\infty}{f_j}(y)\Big||y_k  |^{-\frac{1}{2}}\duy\\
    &=C\varepsilon_j\int_{\Omega_1}|u(y)|\sum\limits_{k=1}^{d}\Big|\alpha_k+\frac{1}{2}  \Big||y_k|^{-\frac{3}{2}}\duy
    +C\varepsilon_j\int_{\Omega_1}|u(y)|\sum\limits_{k=1}^{d}|y_k|^{-\frac{1}{2}}\duy\\
    &\lesssim\varepsilon,
\end{align*}where we have used the facts that
\begin{equation}\label{eq9}
\left\{
\begin{aligned}
\hspace{-0.5cm}
{\int_{\Omega_1}|u(y)|\sum_{k=1}^{d}|y_k|^{-\frac{1}{2}}\duy<\infty,}\\
{\int_{\Omega_1}|u(y)|\sum_{k=1}^{d}\Big|\alpha_{k}+\frac{1}{2}\Big||y_k|^{-\frac{3}{2}}\duy<\infty,}
\end{aligned}
\right.
\end{equation}and in the third inequality we have used the fact that
\begin{align*}
    |y_k+\theta(x_k-y_k)|\geq |y_k|-\theta|x_k-y_k|=\Big
    (1-\frac{\theta}{2}\Big) |y_k|.
\end{align*}
By taking the supremum over $\varphi $ and
the arbitrariness of $\varepsilon>0$, the theorem can be proved.
\end{proof}

\begin{remark}\label{rem2.1} By computation, we conclude that    the function $u\in \bv$  satisfies
(\ref{eq9})   in Theorem \ref{th2.5} when $d\ge 3$, at this time,
Theorem \ref{th2.5} is valid for any open set $\Omega\subseteq
\mathbb{R}^{d}$.
\end{remark}

Moreover, by Lemma \ref{lem 1} and Theorem \ref{th2.5}, we have the
following max-min property of the $\la$-variation.
\begin{theorem}\label{theorem2}
Let $\Omega_1$ be an open set defined in (\ref{eq3}). Suppose $u,v\in L^1(\Omega_1,\da)$, then
    \begin{equation*}
        |\nl\max\{u,v\}
        |(\Omega_1)+|\nl\min\{u,v\}|(\Omega_1)\le |\nl
        u|(\Omega_1)+|\nl v|(\Omega_1).
    \end{equation*}
\end{theorem}
\begin{proof}
Without loss of generality, we may assume
\begin{equation*}
    |\nl u|(\Omega_1)+|\nl v|(\Omega_1)<\infty.
\end{equation*}

By Theorem \ref{th2.5}, we take two functions
\begin{equation*}
    u_h, v_h\in\bu\cap C_c^\infty (\Omega_1),\ \ h=1,2,...,
\end{equation*}
such that
\begin{equation*}
\begin{cases}
u_{h}\rightarrow
u, v_{h}\rightarrow v\ \ \hbox{in}\ \ L^1(\Omega_1,d\mu_{\al}),\\
\int_{\Omega_1} |\nl u_{h}(x)|\du\rightarrow|\nl u|(\Omega_1),\\
\int_{\Omega_1} |\nl v_{h}(x)|\du\rightarrow|\nl v|(\Omega_1).
\end{cases}
\end{equation*}
Since
\begin{equation*}
    \max\{u_{h},v_{h}\}\rightarrow \max\{u,v\}\ \ \&\ \
    \min\{u_{h},v_{h}\}\rightarrow \min\{u,v\}\ \ \hbox{in}\ \
    L^1(\Omega_1,d\mu_{\al}).
\end{equation*}
Via Lemma \ref{lem 1}, it follows that
\begin{align*}
    |&\nl\max\{u,v\}|(\Omega_1)+|\nl\min\{u, v\}|(\Omega_1)\\
    &\le\mathop{\lim\inf}\limits_{h\rightarrow\infty}\int_{\Omega_1}
    |\nl\max\{u_{h},v_{h}\}|\du+\mathop{\lim\inf}\limits_{h\rightarrow\infty}\int_{
    \Omega_1}|\nl \min\{u_{h},v_{h}\}|\du\\
    &\le \mathop{\lim\inf}\limits_{h\rightarrow\infty}
    \bigg(\int_{\Omega_1}|\nl \max\{u_{h},v_{h}\}|\du+ \int_{\Omega_1}|\nl
    \min\{u_{h},v_{h}\}|\du\bigg)\\
    &\le\mathop{\lim\inf}\limits_{h\rightarrow\infty}
    \bigg(\int_{\{x\in\Omega_1: u_{h}\le v_{h}\}}
    |\nl v_{h} |\du+\int_{\{x\in\Omega_1: u_{h}> v_{h}\}}
    |\nl u_{h} |\du\\&\quad+\int_{\{x\in\Omega_1: u_{h}\le v_{h}\}}
    |\nl u_{h} |\du+\int_{\{x\in\Omega_1: u_{h}> v_{h}\}}
    |\nl v_{h} |\du\bigg)\\
    &=\mathop{\lim\inf}\limits_{h\rightarrow\infty}\int_{\Omega_1}
    |\nl u_{h}(x)|\du+\mathop{\lim\inf}\limits_{h\rightarrow\infty}\int_{\Omega_1}
    |\nl v_{h}(x)|\du\\
    &\le \lim_{h\rightarrow\infty}\int_{\Omega_1}
    |\nl u_{h}(x)|\du+\lim_{h\rightarrow\infty}\int_{\Omega_1}
    |\nl v_{h}(x)|\du\\
    &=|\nl u|(\Omega_1)+|\nl v|(\Omega_1).
\end{align*}
\end{proof}

\subsection{Basic properties of Laguerre perimeter}\label{sub2.2}
In this subsection, we introduce a kind of new perimeter: the
Laguerre perimeter ($\la$-perimeter in short). Moreover,
we establish the related results for it.

The $\la$-perimeter of $E\subset \Omega$ can be defined
as follows:
\begin{equation}\label{def-hp}
P_{\la}(E,\Omega)=|\nabla_{\la}1_E|(\Omega)=\sup_{\varphi\in\mathcal {F}(\Omega)}\Big\{\int_E \mathrm{div}_{\la}\varphi(x)\du\Big\},
\end{equation}
where $\mathcal F(\Omega)$ is defined in Section \ref{sub2.1}. In particular,
we shall also write $$P_{\la}(E,\R^d_+)=P_{\la}(E)$$.

The following conclusion is a direct corollary of Lemma \ref{lem 1} to replace $f$ with $1_{E}$.
\begin{corollary}(Lower semicontinuity of $P_{\la}$).\label{semicontinuity-2}
    Suppose $1_{E_k}\to 1_E$ in $L_{loc}^1(\Omega,\da)$, where $E$ and
    $E_k$, $k\in\mathbb N,$ are subsets of $\Omega $, then
    \begin{equation*}
        P_{\la}(E,\Omega)\le\mathop {\lim\inf}\limits_{k\to\infty}P_{\la}(E_k,\Omega).
    \end{equation*}
\end{corollary}

Moreover, by Theorem \ref{theorem2}, via choosing $u=1_E$ and
$v=1_F$ for any compact subsets $E, F$ in $\Omega_1$, we immediately
obtain the following corollary. According to Xiao and Zhang's result
in \cite[Section 1.1 (iii)]{XZ},   we also give the  equality
condition of (\ref{eq8}).
\begin{corollary}\label{coro2.2}
    For any compact subsets $E, F$ in $\Omega_1$, we have
    \begin{equation}\label{eq8}
        P_{\la}(E\cap F,\Omega_1)+P_{\la}(E\cup F,
        \Omega_1)\le P_{\la}(E,\Omega_1)+P_{\la}
        (F,\Omega_1),
    \end{equation}where $\Omega_1$ is an open set defined in
    (\ref{eq3}).
Especially, if $P_{\la}(E\setminus (E\cap F),\Omega_1) \cdot
P_{\la}(F\setminus (F\cap E),\Omega_1)=0$, the equality of
(\ref{eq8}) holds true.  \end{corollary}

\begin{proof}Since (\ref{eq8}) is valid, we only need to prove its converse
inequality  holds true under the above condition. Obviously, the
condition $P_{\la}(E\setminus (E\cap F),\Omega_1)\cdot P_{\la}(F\setminus (E\cap F),\Omega_1)=0$ implies that $P_{\la}(E\setminus (E\cap F),\Omega_1)=0$ or $P_{\la}(F\setminus (E\cap F),\Omega_1)=0$. Suppose $P_{\la}(E\setminus (E\cap F),\Omega_1)=0$. Via (\ref{eq8}), we have
\begin{align}\label{eq11}
P_{\la}(E,\Omega_1)
&=P_{\la}((E\setminus (E\cap F))\cup (E\cap F),\Omega_1)\\ \nonumber
&\le P_{\la}(E\setminus (E\cap F),\Omega_1)+P_{\la}( E\cap F,\Omega_1)\\ \nonumber
&=P_{\la}(E\cap F,\Omega_1).
\end{align}
Using (\ref{def-hp}) and $E\cup F=F\cup (E\setminus (E\cap F))$, we
obtain
\begin{align}\label{eqq12}
P_{\la}(F,\Omega_1)
&=\sup_{\varphi\in\mathcal {F}(\Omega_1)}
\Big\{\int_F \mathrm{div}_{\la}\varphi(x)\du\Big\}\\ \nonumber
&=\sup_{\varphi\in\mathcal {F}(\Omega_1)}\Big\{\int_{E\cup F}
\mathrm{div}_{\la}\varphi(x)\du-\int_{E\setminus (E\cap F)}
\mathrm{div}_{\la}\varphi(x)\du\Big\}\\ \nonumber
&\le\sup_{\varphi\in\mathcal {F}(\Omega_1)}\Big\{\int_{E\cup F}
\mathrm{div}_{\la}\varphi(x)\du\Big\}\\&\quad+\sup_{\varphi\in\mathcal {F}(\Omega_1)} \Big\{\int_{E\setminus (E\cap F)}\mathrm{div}_{\la}\varphi(x)\du\Big\}\\ \nonumber
&=P_{\la}(E\cup F,\Omega_1)+ P_{\la}(E\setminus (E\cap F),\Omega_1)\\\nonumber
&=P_{\la}(E\cup F,\Omega_1).
\end{align}
Combining (\ref{eq11}) with (\ref{eqq12}) deduces that
\begin{equation*}
P_{\la}(E,\Omega_1)+P_{\la}(F,\Omega_1)\le P_{\la}(E\cup
F,\Omega_1)+P_{\la}( E\cap F,\Omega_1),
\end{equation*}which derives the desired result. Another case can be
similarly proved, we omit the details.
\end{proof}

Next we show that the Gauss-Green formula is valid on sets of finite
$\la$-perimeter.
\begin{theorem}(Gauss-Green formula).\label{thm-3} Let $E\subseteq\Omega$ be subset with finite $\la$-perimeter. Then we have
\begin{align*}
    &\int_{E}\mathrm{\widetilde{div}}_{\la}\varphi\left(x\right)d\mu_{\al}\left(x\right)\\
    =&-\int_{\partial{E^c}}(\sqrt{x_1}\varphi_{1}(x),\cdots,\sqrt{x_d}\varphi_{d}(x))\cdot\vec{n}\omega(x)d\mathcal{H}^{d-1}(x)\\&
    -\int_{E}\sum_{i=1}^{d}\frac{\al_i+\frac{1}{2}-x_i}{\sqrt{x_i}}\varphi_{i}(x)\omega(x)dx,
\end{align*}
where the unit vector $\vec{n}(x)$ is the outward normal to $E$ and
$\mathrm{\widetilde{div}}_{\la}(\cdot)$ is defined in
(\ref{eqq7}).
\end{theorem}
\begin{proof} By calculating, we have
    \begin{align*}
    &\int_{E}\mathrm{\widetilde{div}}_{\la}\varphi\left(x\right) d\mu_{\al}\left(x\right)\\
    &=\int_{E}\left(\sum_{i=1}^{d}\sqrt{x_i}\frac{\partial}{\partial{x_i}}\varphi_{i}(x)\right)\omega(x)dx\\
    &=\int_{E}\mathrm{div}(\sqrt{x_1}\varphi_{1}(x)\omega(x),\cdots,\sqrt{x_d}\varphi_{d}(x)\omega(x)     )dx-\int_{E}\sum_{i=1}^{d}\sqrt{x_i}\varphi_{i}(x)\frac{\partial}{\partial{x_i}}\omega(x)dx\\
    &\quad-\frac{1}{2}\int_{E}\sum_{i=1}^{d}\frac{1}{\sqrt{x_i}}\varphi_{i}(x)\omega(x)dx\\
    &=-\int_{\partial{E^c}}(\sqrt{x_1}\varphi_{1}(x),\cdots,\sqrt{x_d}\varphi_{d}(x))\cdot\vec{n}\omega(x)d\mathcal{H}^{d-1}(x)\\&\quad-\int_{E}\sum_{i=1}^{d}\sqrt{x_i}\varphi_{i}(x)\frac{\partial}{\partial{x_i}}\omega(x)dx
    -\frac{1}{2}\int_{E}\sum_{i=1}^{d}\frac{1}{\sqrt{x_i}}\varphi_{i}(x)\omega(x)dx\\
    &=-\int_{\partial{E^c}}(\sqrt{x_1}\varphi_{1}(x),\cdots,\sqrt{x_d}\varphi_{d}(x))\cdot\vec{n}\omega(x)d\mathcal{H}^{d-1}(x)\\
    &\quad-\int_{E}\sum_{i=1}^{d}\frac{\al_i+\frac{1}{2}-x_i}{\sqrt{x_i}}\varphi_{i}(x)\omega(x)dx,
    \end{align*}
    where we have used the classical Gauss-Green formula and the following facts for the derivatives of $\omega(x)$:
    \begin{align*}
    \frac{\partial}{\partial{x_i}}\left(\prod_{j=1}^{d}\frac{{x_j}^{\alpha_j}e^{-x_j}}{\Gamma(\alpha_j+1)}\right)
    &=\prod_{j=1,j\neq i}^{d}\frac{{x_j}^{\alpha_j}e^{-x_j}}{\Gamma(\alpha_j+1)}\frac{1}{\Gamma(\alpha_i+1)}(-e^{-x_i}{x_i}^{\al_i}+\al_i e^{-x_i}{x_i}^{\al_i-1})\\
    &=\Big(-1+\frac{\al_i}{x_i}\Big)\omega(x)
    \end{align*}
for $1\le i\le d$. This completes the proof.
\end{proof}

\begin{lemma}\label{lem-5.1}
    For any set $E$ in $\Omega$ with finite $\la$-perimeter, then
    \begin{equation*}
       {P}_{\la}(E,\Omega)={P}_{\la}(\Omega\backslash E, \Omega).
    \end{equation*}
\end{lemma}
\begin{proof} For any $\varphi\in \mathcal {F}(\R_{+}^{d})$, since ${P}_{\la}(E,\Omega)<\infty$,
then $$\sup_{\varphi\in\mathcal
{F}(\R_{+}^{d})}\int_{E}{\mathrm{div}}_{\la}\varphi(x)\du<\infty.$$
Via the extended Gauss-Green formula (Theorem \ref{thm-3}) and
noting the compact support of $\varphi$, we have
\begin{align*}
  &\int_E\mathrm{div}_{\la}\varphi(x)\du\\
  &=-\int_E\widetilde{\mathrm{div}}_{\la} (\varphi_1(x),\ldots,\varphi_d(x))\du-
  \int_{E}\sum^d_{k=1}\frac{\al_i+\frac{1}{2}-x_i}{\sqrt{x_i}}\varphi_i(x)\du\\
  &=\int_{\partial{E^c}}(\sqrt{x_1}\varphi_{1}(x),\cdots,\sqrt{x_d}\varphi_{d}(x))\cdot\vec{n}\omega(x)d\mathcal{H}^{d-1}(x)\\
  &\quad+\int_{E} \sum^d_{k=1}\frac{\al_i+\frac{1}{2}-x_i}{\sqrt{x_i}}\varphi_i(x)\du-\int_{E} \sum^d_{k=1}\frac{\al_i+\frac{1}{2}-x_i}{\sqrt{x_i}}\varphi_i(x)\du\\
  &=\int_{\partial{E}}(\sqrt{x_1}\varphi_{1}(x),\cdots,\sqrt{x_d}\varphi_{d}(x))\cdot\vec{n}\omega(x)d\mathcal{H}^{d-1}(x)\\
  &=-\int_{E^c}\widetilde{\mathrm{div}}_{\la}\varphi(x)\du-\int_{E^c} \sum^d_{k=1}\frac{\al_i+\frac{1}{2}-x_i}{\sqrt{x_i}}\varphi_i(x)\du\\
  &=\int_{E^c}\mathrm{div}_{\la}\varphi(x)\du,
  \end{align*}
where $\vec{n}(x)$ is the unit exterior normal to $E$ at $x$. Due to the arbitrariness
of $\varphi$, taking the supremum reaches
$${P}_{\la}( E,\Omega)={P}_{\la}(\Omega\backslash E,\Omega).$$
\end{proof}

\subsection{Coarea formula of $\la$-BV functions and the Sobolev inequality}\label{subsec3.3}
Below we prove the coarea formula and the Sobolev inequality for
$\la$-perimeter.
\begin{theorem}\label{coarea formula}
   Let $\Omega_1$ be an open set defined in (\ref{eq3}). If $f\in\bu$, then
    \begin{equation}\label{coarea formula-1}
    |\nabla_{\la}f|(\Omega_1)\approx\int_
    {-\infty}^{+\infty}P_{\la}(E_t,\Omega_1)dt,
    \end{equation}
where $E_t=\{x\in\Omega_1: f(x)>t\}$ for $t\in\mathbb{R}$.
\end{theorem}
\begin{proof}
    Firstly, assume
    \[\varphi=(\varphi_1,\varphi_2,\ldots,\varphi_d)\in C_c^1(\Omega_1,\R^{d}).\] We can
    easily prove that for $i=1,2,\ldots,d$,
    $$\int_{\Omega_1} f(x)\sqrt{x_i}\frac{\partial}{\partial{x_i}}\varphi_i(x)\du=\int_{ -\infty}^{+\infty}\Big(\int_{E_t}\sqrt{x_i}\frac{\partial}{\partial{x_i}}\varphi_i(x)\du\Big)dt,$$
 and
    $$\int_{\Omega_1}
    f(x)\frac{\alpha_{i}+\frac{1}{2}-x_i}{\sqrt{x_i}}\varphi_i(x)\du=
    \int_{-\infty}^{+\infty}\Big(\int_{E_t}\frac{\alpha_{i}+\frac{1}{2}-x_i}{\sqrt{x_i}}\varphi_{i}(x)\du\Big)dt,$$
where the latter can be seen in the proof of \cite [Section 5.5, Theorem 1]{EG}. It follows that
    $$\int_{\Omega_1}f(x)\mathrm {div}_{\la}\varphi(x)
    \du=\int^{+\infty}_{-\infty}\Big(\int_{E_t}
    \mathrm{div}_{\la}\varphi(x)\du\Big)dt.$$
    Therefore, we conclude that for all $\varphi\in
    \mathcal{F}(\Omega_1)$,
    \[\int_{\Omega_1}f(x)\mathrm {div}_{\la}\varphi(x)
    \du \le\int_{-\infty}^{+\infty}P_{\la}(E_t,\Omega_1
    )dt.\] Furthermore, \[|\nabla_{\la}f
    |(\Omega_1)\le\int_{-\infty}^{+\infty}{P_{\la}(E_t,\Omega_1)dt}.
    \]

    Secondly, without loss of generality, we only need to verify that
    \[|\nabla_{\la}f|(\Omega_1)\ge\int_{-\infty}^{+\infty}{P_{\la}(E_t,\Omega_1)dt}
    \] holds for $f\in\mathcal{BV}_{\la}(\Omega_1)\bigcap C^\infty(\Omega_1)$. This proof can refer to the idea of
    \cite[Proposition 4.2]{M}. Let
    \begin{equation*}\label{eqq2.8}
    m(t)=\int_{\{x\in\Omega_1:\ f(x)\le t\}}\Big|\sum_{i=1}^{d}\sqrt{x_i}\frac{\partial }{\partial{x_i}}f(x)\Big|\du.
    \end{equation*}
    It is obvious that
    \begin{equation*}
    \int^{+\infty}_{-\infty}m'(t)dt=\int_{\Omega_1}\Big|\sum_{i=1}^{d}\sqrt{x_i}\frac{\partial }{\partial{x_i}}f(x)\Big|\du.
    \end{equation*}

    Define the following function $g_h$ as
    $$g_h(s):=
    \begin{cases}
    0, &\mathrm{if} \ s\le t,\\
    h(s-t), &\mathrm{if}\ t\le s\le t+1/h,\\
    1, &\mathrm{if} \ s\ge t+1/h,
    \end{cases}$$
    where $t\in \mathbb{R}$. Set the sequence ${v_h}(x):=g_h(f(x))$.
    At this time, ${v_h} \to {1_{{E_t}}}$ in $L^1(\Omega_1,d\mu_{\al})$. In fact,
    \begin{align*}
    \int_{\Omega_1} |v_h(x)-1_{E_t}|\du
    &=\int_{\{x\in\Omega_1:t<f(x)\le t + 1/h\}}|g_h(f(x))-1|\du\\
    &\le \int_{\{x\in\Omega_1:t<f(x)\le t + 1/h\}}\du\to 0.
    \end{align*}
    Since $\{x\in\Omega_1: t<f(x)\le t+1/h\}\rightarrow\emptyset$ as
    $h\rightarrow \infty$, we obtain
    \begin{align*}
        &\int_{\Omega_1}|\nabla_{\la}{v_h}(x)|\du\\
        &=\int_{\{x\in\Omega_1:t<f(x)\le t + {1/h}\}}|\nabla_{\la}(h(f(x)-t))|\du\\ &\quad
        +\int_{\{x \in\Omega_1:f(x)\ge t+1/h\}}|\nabla_{\la}1|\du\\
        &= h\int_{\{x\in\Omega_1:t<f(x)\le t + {1/h}\}}
        \Big|\sum_{i=1}^{d}\sqrt{x_i}\frac{\partial}{\partial{x_i}}f(x)\Big|\du.\\
    \end{align*}
    Taking the limit with $h \rightarrow \infty$ and using Theorem \ref
    {th2.5}, we obtain
    \begin{equation}
        \begin{split}\label{eq-3.10}
        |\nabla_{\la}1_{E_t}|(\Omega_1)
        &\le\mathop{\lim \sup}\limits_{h \to \infty} \int_{\Omega_1}|
        \nabla_{\la}v_h(x)|\du\\
        &=h\mathop{\lim \sup}\limits_{h \to \infty}\int_{\{x\in\Omega_1:t<f(x)\le t + {1/h}\}}
        \Big|\sum_{i=1}^{d}\sqrt{x_i}\frac{\partial}{\partial{x_i}}f(x)\Big|\du\\
        &=m'(t).
        \end{split}
    \end{equation}
    Integrating (\ref {eq-3.10}) reaches
    \begin{align*}
    \int_{-\infty }^{+\infty}P_{\la}(E_t,\Omega_1)dt
    &\le \int_{-\infty}^{+\infty}m'(t)dt\\
    &=\int_{\Omega_1}\Big|\sum_{i=1}^{d}\sqrt{x_i}\frac{\partial}{\partial{x_i}}f(x)\Big|\du\\
    &\lesssim\int_{\Omega_1}|\nabla_{\la}f(x)|\du.
    \end{align*}

Finally, by approximation and using the lower semicontinuity of the
$\la$-perimeter, we conclude that (\ref{coarea formula-1}) holds
true for all $f\in\mathcal {BV}_{\la}(\Omega_1)$.
\end{proof}

Finally, we can develop the Sobolev inequality and the isoperimetric
inequality for $\la$-BV functions. Since the domain $\Omega_{1}$ is
a reasonable substitute of $\Omega$, we can obtain the isoperimetric
inequality and the  Sobolev inequality for
$f\in\mathcal{BV}_{\la}(\Omega_{1})$, where $\Omega_{1}$ is given in
(\ref{eq3}).
\begin{theorem}\label{thm2.7}
\item{\rm (i)} (Sobolev inequality). Let $\Omega_1$ be an open set defined in (\ref{eq3}).
Then for all $f\in {\mathcal{BV}}_{\la}(\Omega_1)$, we have
\begin{equation}\label{eq12}
\|f\|_{L^{\frac{d}{d-1}}(\Omega_1,d\mu_{\al})}\lesssim |\nl
f|(\Omega_1).
\end{equation}

\item{\rm (ii)} (Isoperimetric inequality). Let $E$ be a bounded set of finite $\la$-perimeter in $\Omega_1$. Then
\begin{equation}\label{eq14}
{\mu_{\al}(E)}^{\frac{d-1}{d}}\lesssim
P_{\la}(E,\Omega_1).
\end{equation}
\item{\rm (iii)} The above two statements are equivalent.
\end{theorem}
\begin{proof}
(i) Choose
\begin{equation*}
    f_k\in C^\infty_c(\Omega_{1})\cap
    \mathcal{BV}_{\la}(\Omega_1), \ k=1,2,\ldots,
\end{equation*}
such that
\begin{equation*}
\begin{cases}
f_k\rightarrow f\ \hbox{in}\ L^1(\Omega_1,d\mu_{\al}),\\
\int_{\Omega_1}|\nabla_{\la}
f_k(x)|\du\rightarrow
\parallel\nabla_{\la}f\parallel(\Omega_1).
\end{cases}
\end{equation*}
Since $\Omega_{1}=\Omega\setminus \{x\in\mathbb{R}_{+}^d:\exists i\in 1,\cdots,d\ \mathrm{such}\ \mathrm{that}\ \sqrt{x_i}<1\}$,
 then for any $i=1,\cdots,d$, we obtain $\sqrt{x_i}\geq 1$. It is
 easy to see that
\begin{equation}\label{eq15}
|\nabla f(x)|\le|\nl f(x)|=\Bigg(\sum_{i=1}^{d}\Big(\sqrt{x_i}\frac{\partial}{\partial{x_i}}f(x)\Big)^2\Bigg)^{\frac{1}{2}}.
\end{equation}
Then by Fatou's lemma and the weighted Gagliardo-Nirenberg-Sobolev
inequality, we have
\begin{align*}
\|f\|_{L^{\frac{d}{d-1}}(\Omega_{1},d\mu_{\al})}
&\le\liminf_{k\rightarrow\infty}\|f_k\|_{L^{\frac{d}{d-1}}(\Omega_{1},d\mu_{\al})}\\
&\lesssim\lim_{k\rightarrow\infty}\|\nabla f\|_{L^1(\Omega_{1},d\mu_{\al})}\\
&\lesssim\lim_{k\rightarrow\infty}\|\nabla_{\la}f\|_{L^1(\Omega_{1},d\mu_{\al})}=|\nl f|(\Omega_{1}),
\end{align*}
where we have used the relation between the gradient $\nabla$ and the Laguerre gradient $\nl$ in (\ref{eq15}).

(ii) We can show that (\ref{eq14}) is valid via letting $f=1_E$ in (\ref{eq12}).

(iii) Apparently, (i)$\Rightarrow$(ii) has been proved.
In what follows, we prove (ii)$\Rightarrow$(i). Assume that $0\le f\in C^\infty_c(\Omega_1)$. By the coarea formula in Theorem \ref {coarea formula} and (ii), we have
\begin{equation*}
 \int_{\Omega_1} |\nabla_{\la}f(x)|\du
= \int_{0}^{+\infty}|
 \nabla_{\la} 1_{E_t}|(\Omega_1)\,dt
 \gtrsim\int_{0}^{+\infty}|E_t|^{\frac{d-1}{d}}dt,
 \end{equation*}
where $E_t=\big\{x\in\Omega_1:\ f(x)>t\big\}$. Let
\begin{equation*}
f_t=\min\{t,f\}\ \ \&\ \ \chi (t) =\left(\int_{\Omega_1} f_t^{\frac{d}{d-1}}(x)\du\right)^{\frac{d-1}{d}},\  \forall\ t\in\mathbb{R}.
\end{equation*}
It is easy to see that
\begin{equation*}
\lim_{t \rightarrow \infty}\chi(t)=\bigg(\int_{\Omega_1}
|f(x)|^{\frac{d}{d-1}}\du\bigg)^{\frac{d-1}{d}}.
\end{equation*}

In addition, we can check that $\chi(t)$ is nondecreasing on
$(0,\infty)$ and for $h>0$,
\begin{equation*}
0\le \chi(t+h)-\chi(t)\le \bigg(
\int_{\Omega_1}
|f_{t+h}(x)-f_t(x)|^{\frac{d}{d-1}}\du\bigg)^{\frac{d-1}{d}}\le
h|E_t|^{\frac{d-1}{d}}.
\end{equation*}
Then $\chi(t)$ is locally a Lipschitz function and
$\chi'(t)\le |E_t|^{\frac{d-1}{d}}$ for a.e. $t\in (0,\infty)$.
Hence,
\begin{align*}
\bigg(\int_{\Omega_1}|f(x)|^{\frac{d}{d-1}}\du\bigg)^
{\frac{d-1}{d}}=\int^\infty_0 \chi'(t)dt&\le\int^\infty_0 |E_t|^{\frac{d-1}{d}}dt\\
&\lesssim\int_{\Omega_1}|\nabla_{\la} f(x)|\du.
\end{align*}
For all $f\in {\mathcal{BV}}_{\la}(\Omega_1)$, we conclude that (\ref{eq12}) is valid by Theorem \ref{th2.5}.
\end{proof}

As a direct result of the proof of (i) in Theorem \ref{thm2.7}, we
can get the following corollary.
\begin{corollary}\label{prop2} Let $1<p<d$ and let $\Omega_1$ be an open set defined in (\ref{eq3}). For any $f\in W_{\la}^{1,1}(\Omega_{1})$ one has
\begin{equation}\label{equa19}
\|f\|_{L^{\frac{dp}{d-p}}(\Omega_{1},d\mu_{\al})}\lesssim\|\nl f\|_{L^p(\Omega_{1},d\mu_{\al})}.
\end{equation}

\end{corollary}
\begin{proof}For some $\gamma>1$ to be fixed later, via the Lemma \ref{lem 1} (i) and H\"{o}lder inequality we obtain
\begin{align*}
\Bigg(\int_{\Omega_1}&|f(x)|^{\frac{\gamma d}{d-1}}\du\Bigg)^{\frac{d-1}{d}}\\
&\lesssim\int_{\Omega_1}|
f(x)|^{\gamma-1}|\nabla_{\la} f(x)|\du\\
&\lesssim
\left(\int_{\Omega_1}|f(x)|^{\frac{p(\gamma-1)}{p-1}}\du\right)^{1-\frac{1}{p}}
\left(\int_{\Omega_1}|\nabla_{\la}f(x)|^p \du\right)^{\frac{1}{p}}.
\end{align*}
Choosing
$$\gamma=\frac{p(d-1)}{d-p}$$ and noting
$$\gamma-1=\frac{d(p-1)}{d-p},
$$
then we conclude that (\ref{equa19}) holds true.

\end{proof}

\section{Laguerre mean curvature}\label{sec-5}
In this section  we focus on the question whether  every set of
finite $\la$-perimeter in $\Omega_1\subseteq\mathbb R_{+}^{d}$ has
mean curvature in $L^{1}(\Omega_1,d\mu_{\al})$. For the classical
case, please refer to \cite{Barozzi} for details.  During the proof
of Theorem \ref{thm5-1}, we need to use the important result  for
the Laguerre perimeter in Corollary \ref{coro2.2}. Therefore, we
assume that the dimension $d\ge 3$ via Remark \ref{rem2.1}.

For a given $u\in L^{1}(\Omega_1,d\mu_{\al})$, the Massari type
functional corresponding to the $\la$-perimeter is defined as
\begin{equation*}\label{eqq5.1}
\mathscr{F}_{u,\la}(E):={P}_{\la}(E,\Omega_1)+\int_{E}u(x)\du,
\end{equation*}
where $E$ is an arbitrary set of finite $\la$-perimeter in $\mathbb
R_{+}^{d}$.

\begin{theorem}\label{thm5-1}
    For every set $E$ of finite $\la$-perimeter in $\mathbb
    R_{+}^{d}$, there exists a function $u\in L^{1}(\mathbb R_{+}^{d},d\mu_{\al})$ such that
    $$\mathscr{F}_{u,\la}(E)\leq \mathscr{F}_{u,\la}(F)$$
    holds for every set $F$ of finite $\la$-perimeter in $\Omega_1$.
\end{theorem}
\begin{proof}
    At first, for a given set $E$, we need to find a function $u\in
    L^{1}(\Omega_1,d\mu_{\al})$ such that
    \begin{equation}\label{eq5.1}
    \mathscr{F}_{u,\la}(E)\leq \mathscr{F}_{u,\la}(F)
    \end{equation}
    holds for every $F$ with either $F\subset E$ or $E\subset F$, then
    Theorem \ref{thm5-1} is proved, i.e. (\ref{eq5.1}) holds for every
    $F\subset \Omega_1$. In fact, by adding the inequality
    (\ref{eq5.1}) corresponding to the test sets $E\cap F$ and $E\cup
    F$,
    $$\begin{cases}
    {P}_{\la}(E,\Omega_1)+\int_{E}u(x)\du\leq {P}_{\la}(E\cap F,\Omega_1)+\int_{E\cap F}u(x)\du,\\
    {P}_{\la}(E,\Omega_1)+\int_{E}u(x)\du\leq
   {P}_{\la}(E\cup F,\Omega_1)+\int_{E\cup F}u(x)\du.
    \end{cases}$$
    Then noting that
    \begin{equation}\label{eq5.3}
    {P}_{\la}(E\cap
    F,\Omega_1)+{P}_{\la}(E\cup F,\Omega_1)\le
    {P}_{\la}(E,\Omega_1)+{P}_{\la}(F,\Omega_1),
    \end{equation}
    we can get
    \begin{align*}
        &2{P}_{\la}(E,\Omega_1)+2\int_{E}u(x)\du\\
        &\leq{P}_{\la}(E\cap F,\Omega_1)+{P}_{\la}(E\cup F,\Omega_1)
        +\int_{E\cap F}u(x)\du+\int_{E\cup F}u(x)\du\\
        &\leq{P}_{\la}(E,\Omega_1)+{P}_{\la}(F,\Omega_1)+\int_{E}u(x)\du+\int_{F}u(x)\du,
    \end{align*}
    that is, (\ref{eq5.1}) holds for arbitrary $F$. Also, if
    (\ref{eq5.1}) holds for $F\subset E$, then   for the set  $F$ such
    that $E\subset F$, i.e. $ \Omega_1\backslash F \subset  \Omega_1\backslash E$,
    \begin{align*}
        &{P}_{\la}(E,\Omega_1)+\int_{E}u(x)\du\\
        &={P}_{\la}(\Omega_1\backslash E, \Omega_1)+\int_{\Omega_1\backslash E}u(x)\du-\int_{\Omega_1\backslash E}u(x)\du+\int_{E}u(x)\du\\
        &\leq{P}_{\la}(\Omega_1\backslash F, \Omega_1)+
        \int_{\Omega_1\backslash F}u(x)\du-\int_{\Omega_1\backslash E}u(x)\du+\int_{E}u(x)\du\\
        &={P}_{\la}(F,\Omega_1)
        +\int_{\Omega_1\backslash  F}u(x)\du-\int_{\Omega_1\backslash  E}u(x)\du+\int_{E}u(x)\du\\
        &=\mathscr{F}_{u,\la}(F)-\int_{F}u(x)\du
        +\int_{\Omega_1\backslash F}u(x)\du-\int_{\Omega_1\backslash  E}u(x)\du\\&\quad+\int_{E}u(x)\du\\
        &=\mathscr{F}_{u,\la}(F)+\int_{F\backslash
        E}u(x)\du-\int_{(\Omega_1\backslash E)/ (\Omega_1\backslash F)}u(x)\du\\
        &=\mathscr{F}_{u,\la}(F),
    \end{align*}
    where we have used the fact that $u$ vanishes outside the set $E$ and Lemma \ref{lem-5.1}.
    Hence, we only need to prove that $u$ defined on $E$ is integrable and (\ref{eq5.1}) holds for any $F\subset E$.

    {\it Step I.} Denote by  $h(\cdot)$   a measurable function
    satisfying $h>0$ on $E$ and $\int_{E}h(x)\du<\infty$, and denote
    by $\Lambda$ the (positive and totally finite) measure:
    \begin{equation*}
    \Lambda(F)=\int_{F}h(x)\du,\ F\subset E.
    \end{equation*}
    It is obvious that  $\Lambda(F)=0$ if and only if $\mu_{\al}(F)=0$. For $\lambda>0$ and $F\subset E$, consider the functional
    $$\mathscr{F}_{\lambda}(F):={P}_{\la
    }(F,\Omega_1)+\lambda\Lambda(E\setminus F).$$
    It is well known that  every minimizing sequence is compact in
    $L^{1}_{loc}(\Omega_1,d\mu_{\al})$ and the functional is
    lower semi-continuous with respect to the same convergence. Hence, we conclude that, for every $\lambda>0$, there is
     a solution $E_{\lambda}$ to the problem:
    $$\mathscr{F}_{\lambda}(F)\rightarrow \text{min},\ F\subset E.$$
    Choose a sequence $\{\lambda_{i}\}$ of positive numbers which is
    strictly increasing to $\infty$  and denote the corresponding
    solutions by $E_{i}\equiv E_{\lambda_{i}}$, so that $\forall i\geq
    1$,
    \begin{equation}\label{eq5.2}
    \mathscr{F}_{\lambda_{i}}(E_{i})\leq \mathscr{F}_{\lambda_{i}}(F),\ \forall\ F\subset E.
    \end{equation}
    Given $i<j$. Let $F=E_{i}\cap E_{j}$. It follows from (\ref{eq5.2}) that
     $$\mathscr{F}_{\lambda_{i}}(E_{i})\leq \mathscr{F}_{\lambda_{i}}(E_{i}\cap E_{j}),
    $$ that is,
    $${P}_{\la}(E_{i},\Omega_1)+\lambda_{i}\Lambda(E\setminus E_{i})
    \leq {P}_{\la}(E_{i}\cap
    E_{j},\Omega_1)+\lambda_{i}\Lambda(E\setminus(E_{i}\cap E_{j})),$$ which
    implies
    $${P}_{\la}(E_{i},\Omega_1)+\lambda_{i}\int_{E\setminus E_{i}}h(x)\du\leq {P}_{\la}(E_{i}\cap E_{j},\Omega_1
    )+\lambda_{i}\int_{E\setminus(E_{i}\cap E_{j})}h(x)\du.$$
    A direct computation gives
    $${P}_{\la}(E_{i},\Omega_1)\leq\lambda_{i}\int_{E_{i}\setminus E_{j}}h(x)\du+ {P}_{\la}(E_{i}\cap E_{j},\Omega_1).$$
    On the other hand, taking $F= E_{i}\cup E_{j}\subset E$ in (\ref{eq5.2}), we can get
    $\mathscr{F}_{\lambda_{j}}(E_{j})\leq \mathscr{F}_{\lambda_{j}}(E_{i}\cup E_{j})$.
    Hence,
    $${P}_{\la}(E_{j},\Omega_1)+\lambda_{j}\int_{E\setminus E_{j}}h(x)\du\leq {P}_{\la}(E_{i}\cup E_{j},\Omega_1)
    +\lambda_{j}\int_{E\setminus(E_{i}\cup E_{j})}h(x)\du,$$
    equivalently,
    $${P}_{\la}(E_{j},\Omega_1)+\lambda_{j}\int_{E_{i}\setminus E_{j}}h(x)\du\leq {P}_{\la}(E_{i}\cup E_{j},\Omega_1)$$
    which implies that
    \begin{align*}
        {P}_{\la}(E_{i},\Omega_1)&+{P}_{\la}(E_{j},\Omega_1)+\lambda_{j}\int_{E_{i}\setminus E_{j}}h(x)\du\\ &\le
        {P}_{\la}(E_{i}\cup E_{j},\Omega_1)+\lambda_{i}\int_{E_{i}\setminus E_{j}}h(x)\du+ {P}_{\la}(E_{i}\cap E_{j},\Omega_1).
    \end{align*}
    Recall that $h>0$. The above estimate, together with (\ref{eq5.3}) and the fact that $\lambda_{i}<\lambda_{j}$, indicates that
    $$(\lambda_{j}-\lambda_{i})\Lambda(E_{i}\setminus E_{j})=(\lambda_{j}-\lambda_{i})\int_{E_{i}\setminus E_{j}}h(x)\du=0,$$
    that is, $E_{i}\subset E_{j}$ and the sequence of minimizers $\{E_{i}\}$ is increasing. On the other hand, letting $F=E$, we get
    $${P}_{\la}(E_{i},\Omega_1)+\lambda_{i}\Lambda(E\setminus E_{i})\leq {P}_{\la}(E,\Omega_1)+
    \lambda_{i}\Lambda(E\setminus
    E)={P}_{\la}(E,\Omega_1)\ \ \forall\ i\geq 1,$$ which
    deduces that $E_{i}$ converges monotonically and in
    $L^{1}_{loc}(\mathbb R_{+}^{d},d\mu_{\al})$ to $E$. Via Lemma \ref{lem 1} (ii), we get
    $$\begin{cases}
    {P}_{\la}(E,\Omega_1)\leq\mathop{\lim\inf}\limits_{i\rightarrow \infty}{P}_{\la}(E_{i},\Omega_1)\leq {P}_{\la}(E,\Omega_1),\\
    {P}_{\la}(E,\Omega_1)\leq\mathop{\lim\inf}\limits_{i\rightarrow
    \infty}{P}_{\la}(E_{i},\Omega_1)\leq\mathop{\lim\sup}\limits_{i\rightarrow
    \infty}{P}_{\la}(E_{i},\Omega_1)\leq
    {P}_{\la}(E),
    \end{cases}$$
    which means
    \begin{equation}\label{eq5.4}
   {P}_{\la}(E,\Omega_1)=\lim_{i\rightarrow
    \infty}{P}_{\la}(E_{i},\Omega_1).
    \end{equation}

    {\it Step II.} Let $\lambda_{0}=0$ and $E_{0}=\emptyset$, and define
    $$u(x)=
    \begin{cases}
    -\lambda_{i} h(x),\ x\in E_{i}\backslash E_{i-1},\ i\geq1,\\
    0,\ \text{otherwise}.
    \end{cases}$$
    Clearly, $u$ is negative almost everywhere on $E$, and
    \begin{align*}
        \int_{\mathbb R_{+}^{d}}|u(x)|\du
        &=\int_{\cup^{\infty}_{i=0}E_{i+1}\backslash E_{i}}|u(x)|\du\\
        &=\sum^{\infty}_{i=0}\int_{E_{i+1}\backslash E_{i}}\lambda_{i+1} h(x)\du\\
        &=\sum^{\infty}_{i=0}\lambda_{i+1}\Lambda(E_{i+1}\backslash E_{i}).
    \end{align*}
    In (\ref{eq5.2}), taking $F=E_{i+1}$, we have
    $${P}_{\la}(E_{i},\Omega_1)+\lambda_{i}\Lambda(E\setminus E_{i})\leq {P}_{\la}(E_{i+1},\Omega_1)+\lambda_{i}\Lambda(E\setminus E_{i+1}),$$
    that is, for every $i\geq 0$,
    $$\lambda_{i}\Lambda(E_{i+1}\backslash E_{i})\leq {P}_{\la}(E_{i+1},\Omega_1)-{P}_{\la}(E_{i},\Omega_1).$$
    Then for sufficiently large $N$,
    \begin{align*}
        \sum^{N}_{i=0}\lambda_{i}\Lambda(E_{i+1}\backslash E_{i})&\leq \sum^{N}_{i=0}\Big[{P}_{\la}(E_{i+1},\Omega_1)-{P}_{\la}(E_{i},\Omega_1
        )\Big]\\&={P}_{\la}(E_{N},\Omega_1)
        -{P}_{\la}(E_{0},\Omega_1)\\&={P}_{\la}(E_{N},\Omega_1).
    \end{align*}
    Letting $N\rightarrow \infty$, (\ref{eq5.4}) indicates that
    \begin{equation*}
    \sum^{\infty}_{i=0}\lambda_{i}\Lambda(E_{i+1}\backslash E_{i})\leq
    {P}_{\la}(E,\Omega_1).
    \end{equation*}
    We make the additional assumption that $0<\lambda_{i+1}-\lambda_{i}\leq c$, $i\geq 0$, where $c$ is a constant independent of $i$.
    Then for any $N>0$,
    \begin{align*}
        \sum^{N}_{i=0}(\lambda_{i+1}-\lambda_{i})\Lambda(E_{i+1}\backslash E_{i})
        &\leq c\sum^{N}_{i=0}\Lambda(E_{i+1}\backslash E_{i})\\
        &=c\sum^{N}_{i=0}\int_{E_{i+1}\backslash E_{i}}h(x)\du\\
        &=c\int_{\cup^{N}_{i=0}(E_{i+1}\backslash E_{i})}h(x)\du,
    \end{align*}
    which gives
    $$\sum^{\infty}_{i=0}(\lambda_{i+1}-\lambda_{i})\Lambda(E_{i+1}\backslash E_{i})\leq c\Lambda(E).$$
    Then
    \begin{align*}
        \int_{\mathbb R_{+}^{d}}|u(x)|\du
        &=\sum^{\infty}_{i=0}\lambda_{i+1}\Lambda(E_{i+1}\backslash E_{i})\\
        &=\sum^{\infty}_{i=0}(\lambda_{i+1}-\lambda_{i})
        \Lambda(E_{i+1}\backslash E_{i})+\sum^{\infty}_{i=0}\lambda_{i}
        \Lambda(E_{i+1}\backslash E_{i})\\
        &\leq c\Lambda(E)+{P}_{\la}(E,\Omega_1)<\infty.
    \end{align*}
    In conclusion, $u\in L^{1}(\mathbb R_{+}^{d},d\mu_{\al})$.

    {\it Step III.} We claim that for every $i\geq 1$ the inequality
    \begin{equation}\label{eq5.5}
    {P}_{\la}(E_{i},\Omega_1)\leq
    {P}_{\la}(F,\Omega_1)+\sum^{i}_{j=1}\lambda_{j}\Lambda((E_{j}\backslash E_{j-1})\backslash F)
    \end{equation}
    holds for any $F\subset E$.

    For $i=1$, $E_{i-1}=E_{0}=\emptyset$. Then (\ref{eq5.5}) becomes
    $${P}_{\la}(E_{1},\Omega_1)\leq {P}_{\la}(F,\Omega_1)+\lambda_{1}\Lambda(E_{1}\backslash F),$$
    which coincides with (\ref{eq5.2}) for $i=1$.

    Now we assume that (\ref{eq5.5}) holds for a fixed $i\geq 1$ and
    every $F\subset E$. Take $F\cap E_{i}$ as a test set. Note that
    $\{E_{j}\}$ is increasing. It is easy to see that $$(E_{j}\backslash
    E_{j-1})\backslash (F\cap E_{i})=(E_{j}\backslash E_{j-1})\backslash
    F.$$ Then
    \begin{align*}
        {P}_{\la}(E_{i},\Omega_1)
        &\leq{P}_{\la}(F\cap E_{i},\Omega_1)+\sum^{i}_{j=1}\lambda_{j}\Lambda((E_{j}\backslash E_{j-1})\backslash (F\cap E_{i}))\\
        &={P}_{\la}(F\cap
        E_{i},\Omega_1)+\sum^{i}_{j=1}\lambda_{j}\Lambda((E_{j}\backslash
        E_{j-1})\backslash F).
    \end{align*}
    On the other hand, $E_{i+1}$ is a minimizer of $\mathscr{F}_{\lambda_{i+1}}$. Hence,
    $$\mathscr{F}_{\lambda_{i+1}}(E_{i+1})\leq \mathscr{F}_{\lambda_{i+1}}(F\cup E_{i}),$$
    and noticing that $$E\backslash E_{i}=(E\backslash E_{i+1})\cup (E_{i+1}\backslash E_{i}),$$ we can get
    $$E\backslash(F\cup E_{i})= ((E\backslash E_{i+1})\backslash F)\cup ((E_{i+1}\backslash E_{i})\backslash F).$$
    This gives
    \begin{align*}
       {P}_{\la}(E_{i+1},\Omega_1)+\lambda_{i+1}\Lambda(E\backslash E_{i+1})
        &\leq{P}_{\la}(F\cup E_{i},\Omega_1)+\lambda_{i+1}\Lambda(E\backslash(F\cup E_{i}))\\
        &\leq{P}_{\la}(F\cup
        E_{i},\Omega_1)+\lambda_{i+1}\Lambda((E\backslash E_{i+1})\backslash
        F)\\ &\quad+\lambda_{i+1}\Lambda((E_{i+1}\backslash E_{i})\backslash F).
    \end{align*}
    Therefore, we obtain that
    \begin{align*}
        &{P}_{\la}(E_{i},\Omega_1)+{P}_{\la}(E_{i+1},\Omega_1)+\lambda_{i+1}\Lambda(E\backslash E_{i+1})\\
        &\quad\leq {P}_{\la}(F\cap E_{i},\Omega_1)+\sum^{i}_{j=1}\lambda_{j}\Lambda((E_{j}\backslash E_{j-1})\backslash F)\\
        &\quad\quad+{P}_{\la}(F\cup E_{i},\Omega_1)+\lambda_{i+1}\Lambda((E\backslash E_{i+1})\backslash F)+\lambda_{i+1}\Lambda((E_{i+1}\backslash E_{i})\backslash F)\\
        &\quad\leq {P}_{\la}(E_{i},\Omega_1)+{P}_{\la}(F,\Omega_1)+\sum^{i+1}_{j=1}\lambda_{j}\Lambda((E_{j}\backslash E_{j-1})\backslash F)+\lambda_{i+1}\Lambda((E\backslash E_{i+1})\backslash F)\\
        &\quad\leq
       {P}_{\la}(E_{i},\Omega_1)+{P}_{\la}(F,\Omega_1)+\sum^{i+1}_{j=1}\lambda_{j}\Lambda((E_{j}\backslash E_{j-1})\backslash F)+\lambda_{i+1}\Lambda(E\backslash E_{i+1}),
    \end{align*}
    that is, (\ref{eq5.5}) holds for $i+1$. Finally,
    \begin{align*}
        {P}_{\la}(E,\Omega_1)
        &=\lim_{i\rightarrow\infty}{P}_{\la}(E_{i},\Omega_1)\\
        &\leq{P}_{\la}(F,\Omega_1)+
        \lim_{i\rightarrow\infty}\sum^{i}_{j=1}\lambda_{j}\Lambda((E_{j}\backslash E_{j-1})\backslash F)\\
        &={P}_{\la}(F,\Omega_1)-\int_{\cup^{\infty}_{j=0}(E_{j}\backslash E_{j-1})\backslash F}u(x)\du\\
        &={P}_{\la}(F,\Omega_1)-\int_{E\backslash F}u(x)\du,
    \end{align*}
    which gives (\ref{eq5.2}).
\end{proof}

\end{document}